\documentclass[12pt]{amsart}
\usepackage{amssymb,latexsym,tikz}
\usepackage[dvips]{epsfig}
\usepackage[usenames,dvipsnames]{pstricks}
\usepackage{hyperref}

\begin{document}
\title{Resolutions of letterplace ideals of posets}

\author{Alessio D'Al{\`i}}
\address{Dipartimento di Matematica\\
         Universit{\`a} degli Studi di Genova\\
         Via Dodecaneso 35\\
         16146 Genova\\
         Italy}
\email{dali@dima.unige.it}

\author{Gunnar Fl{\o}ystad}
\address{Universitetet i Bergen\\ 
         Matematisk institutt \\
        Postboks 7803\\
        5020 Bergen \\
        Norway} 
\email{gunnar@mi.uib.no}

\author{Amin Nematbakhsh}
\address{School of Mathematics\\
         Institute for Research in Fundamental Sciences (IPM)\\
         P.O. Box: 19395-5746\\
         Tehran\\
         Iran}
\email{nematbakhsh@ipm.ir}


\subjclass[2010]{Primary: 13D02, 05E40, 06A11}
\date{\today}

\begin{abstract}
We investigate resolutions of letterplace ideals of posets.
We develop topological
results to compute their multigraded Betti numbers, and to give structural
results on these Betti numbers. If the poset is a union of no
more than $c$ chains, we show that the Betti numbers may be
computed from simplicial complexes of no more than $c$ vertices.
We also give a recursive procedure to compute the Betti diagrams when
the Hasse diagram of $P$ has tree structure.
\end{abstract}

\maketitle


\theoremstyle{plain}
\newtheorem{theorem}{Theorem}[section]
\newtheorem{corollary}[theorem]{Corollary}
\newtheorem*{main}{Main Theorem}
\newtheorem{lemma}[theorem]{Lemma}
\newtheorem{proposition}[theorem]{Proposition}

\theoremstyle{definition}
\newtheorem{definition}[theorem]{Definition}
\newtheorem{fact}{Fact}

\theoremstyle{remark}
\newtheorem{notation}[theorem]{Notation}
\newtheorem{remark}[theorem]{Remark}
\newtheorem{example}[theorem]{Example}
\newtheorem{claim}{Claim}


\newcommand{\psp}[1]{{{\bf P}^{#1}}}
\newcommand{\psr}[1]{{\bf P}(#1)}
\newcommand{\op}{{\mathcal O}}
\newcommand{\opw}{\op_{\psr{W}}}
\newcommand{\go}{\op}

\newcommand{\ini}[1]{\text{in}(#1)}
\newcommand{\gin}[1]{\text{gin}(#1)}
\newcommand{\kr}{{\Bbbk}}
\newcommand{\pd}{\partial}
\newcommand{\vardel}{\partial}
\renewcommand{\tt}{{\bf t}}


\newcommand{\coh}{{{\text{{\rm coh}}}}}


\newcommand{\modv}[1]{{#1}\text{-{mod}}}
\newcommand{\modstab}[1]{{#1}-\underline{\text{mod}}}

\newcommand{\sut}{{}^{\tau}}
\newcommand{\sumit}{{}^{-\tau}}
\newcommand{\til}{\thicksim}

\newcommand{\totp}{\text{Tot}^{\prod}}
\newcommand{\dsum}{\bigoplus}
\newcommand{\dprod}{\prod}
\newcommand{\lsum}{\oplus}
\newcommand{\lprod}{\Pi}

\newcommand{\La}{{\Lambda}}

\newcommand{\sirstj}{\circledast}

\newcommand{\she}{\EuScript{S}\text{h}}
\newcommand{\cm}{\EuScript{CM}}
\newcommand{\cmd}{\EuScript{CM}^\dagger}
\newcommand{\cmri}{\EuScript{CM}^\circ}
\newcommand{\cler}{\EuScript{CL}}
\newcommand{\clerd}{\EuScript{CL}^\dagger}
\newcommand{\clerri}{\EuScript{CL}^\circ}
\newcommand{\gor}{\EuScript{G}}
\newcommand{\gF}{\mathcal{F}}
\newcommand{\gG}{\mathcal{G}}
\newcommand{\gM}{\mathcal{M}}
\newcommand{\gE}{\mathcal{E}}
\newcommand{\gD}{\mathcal{D}}
\newcommand{\gI}{\mathcal{I}}
\newcommand{\gP}{\mathcal{P}}
\newcommand{\gK}{\mathcal{K}}
\newcommand{\gL}{\mathcal{L}}
\newcommand{\gS}{\mathcal{S}}
\newcommand{\gC}{\mathcal{C}}
\newcommand{\gO}{\mathcal{O}}
\newcommand{\gJ}{\mathcal{J}}
\newcommand{\gU}{\mathcal{U}}
\newcommand{\mm}{\mathfrak{m}}
\newcommand{\cP}{\mathcal P}
\newcommand{\cS}{\mathcal S}
\newcommand{\St}{\mathcal B}
\newcommand{\ovgL}{\overline{\gL}}
\newcommand{\ovS}{\overline{S}}

\newcommand{\dlim} {\varinjlim}
\newcommand{\ilim} {\varprojlim}

\newcommand{\CM}{\text{CM}}
\newcommand{\Mon}{\text{Mon}}


\newcommand{\Kom}{\text{Kom}}


\newcommand{\EH}{{\mathbf H}}
\newcommand{\res}{\text{res}}
\newcommand{\Hom}{\text{Hom}}
\newcommand{\inhom}{{\underline{\text{Hom}}}}
\newcommand{\Ext}{\text{Ext}}
\newcommand{\Tor}{\text{Tor}}
\newcommand{\ghom}{\mathcal{H}om}
\newcommand{\gext}{\mathcal{E}xt}
\newcommand{\id}{\text{{id}}}
\newcommand{\im}{\text{im}\,}
\newcommand{\codim} {\text{codim}\,}
\newcommand{\resol}{\text{resol}\,}
\newcommand{\rank}{\text{rank}\,}
\newcommand{\lpd}{\text{lpd}\,}
\newcommand{\coker}{\text{coker}\,}
\newcommand{\supp}{\text{supp}\,}
\newcommand{\Ad}{A_\cdot}
\newcommand{\Bd}{B_\cdot}
\newcommand{\Fd}{F_\cdot}
\newcommand{\Gd}{G_\cdot}


\newcommand{\sus}{\subseteq}
\newcommand{\sups}{\supseteq}
\newcommand{\pil}{\rightarrow}
\newcommand{\vpil}{\leftarrow}
\newcommand{\rpil}{\leftarrow}
\newcommand{\lpil}{\longrightarrow}
\newcommand{\inpil}{\hookrightarrow}
\newcommand{\pils}{\twoheadrightarrow}
\newcommand{\projpil}{\dashrightarrow}
\newcommand{\dotpil}{\dashrightarrow}
\newcommand{\adj}[2]{\overset{#1}{\underset{#2}{\rightleftarrows}}}
\newcommand{\mto}[1]{\stackrel{#1}\longrightarrow}
\newcommand{\vmto}[1]{\stackrel{#1}\longleftarrow}
\newcommand{\mtoelm}[1]{\stackrel{#1}\mapsto}

\newcommand{\eqv}{\Leftrightarrow}
\newcommand{\impl}{\Rightarrow}

\newcommand{\iso}{\cong}
\newcommand{\te}{\otimes}
\newcommand{\into}[1]{\hookrightarrow{#1}}
\newcommand{\ekv}{\Leftrightarrow}
\newcommand{\equi}{\simeq}
\newcommand{\isopil}{\overset{\cong}{\lpil}}
\newcommand{\equipil}{\overset{\equi}{\lpil}}
\newcommand{\ispil}{\isopil}
\newcommand{\vvi}{\langle}
\newcommand{\hvi}{\rangle}
\newcommand{\susneq}{\subsetneq}
\newcommand{\sgn}{\text{sign}}
\newcommand{\bihom}[2]{\overset{#1}{\underset{#2}{\rightleftarrows}}}


\newcommand{\xd}{\check{x}}
\newcommand{\ortog}{\bot}
\newcommand{\tL}{\tilde{L}}
\newcommand{\tM}{\tilde{M}}
\newcommand{\tH}{\tilde{H}}
\newcommand{\tC}{\tilde{C}}
\newcommand{\tvH}{\widetilde{H}}
\newcommand{\tvh}{\widetilde{h}}
\newcommand{\tV}{\tilde{V}}
\newcommand{\tS}{\tilde{S}}
\newcommand{\tT}{\tilde{T}}
\newcommand{\tR}{\tilde{R}}
\newcommand{\tf}{\tilde{f}}
\newcommand{\ts}{\tilde{s}}
\newcommand{\tp}{\tilde{p}}
\newcommand{\tr}{\tilde{r}}
\newcommand{\tfst}{\tilde{f}_*}
\newcommand{\empt}{\emptyset}
\newcommand{\bfa}{{\bf a}}
\newcommand{\bfb}{{\bf b}}
\newcommand{\bfd}{{\bf d}}
\newcommand{\bfl}{{\bf \ell}}
\newcommand{\la}{\lambda}
\newcommand{\bfen}{{\mathbf 1}}
\newcommand{\ep}{\epsilon}
\newcommand{\en}{r}
\newcommand{\tu}{s}

\newcommand{\ome}{\omega_E}

\newcommand{\bevis}{{\bf Proof. }}
\newcommand{\demofin}{\qed \vskip 3.5mm}
\newcommand{\nyp}[1]{\noindent {\bf (#1)}}
\newcommand{\demo}{{\it Proof. }}
\newcommand{\demodone}{\demofin}
\newcommand{\parg}{{\vskip 2mm \addtocounter{theorem}{1}  
                   \noindent {\bf \thetheorem .} \hskip 1.5mm }}

\newcommand{\lcm}{{\text{lcm}}}


\newcommand{\dl}{\Delta}
\newcommand{\cdel}{{C\Delta}}
\newcommand{\cdelp}{{C\Delta^{\prime}}}
\newcommand{\dlst}{\Delta^*}
\newcommand{\Sdl}{{\mathcal S}_{\dl}}
\newcommand{\lk}{\text{lk}}
\newcommand{\lkd}{\lk_\Delta}
\newcommand{\lkp}[2]{\lk_{#1} {#2}}
\newcommand{\del}{\Delta}
\newcommand{\delr}{\Delta_{-R}}
\newcommand{\dd}{{\dim \del}}
\newcommand{\Del}{\Delta}
\newcommand{\PA}{\cS}
\newcommand{\Stbs}{\St \backslash \St_0}

\renewcommand{\aa}{{\bf a}}
\newcommand{\bb}{{\bf b}}
\newcommand{\cc}{{\bf c}}
\newcommand{\xx}{{\bf x}}
\newcommand{\yy}{{\bf y}}
\newcommand{\zz}{{\bf z}}
\newcommand{\mv}{{\xx^{\aa_v}}}
\newcommand{\mF}{{\xx^{\aa_F}}}

\newcommand{\Symm}{\text{Sym}}
\newcommand{\pnm}{{\bf P}^{n-1}}
\newcommand{\opnm}{{\go_{\pnm}}}
\newcommand{\ompnm}{\omega_{\pnm}}

\newcommand{\pn}{{\bf P}^n}
\newcommand{\hele}{{\mathbb Z}}
\newcommand{\nat}{{\mathbb N}}
\newcommand{\rasj}{{\mathbb Q}}

\newcommand{\dt}{\bullet}
\newcommand{\st}{\hskip 0.5mm {}^{\rule{0.4pt}{1.5mm}}}              
\newcommand{\disk}{\scriptscriptstyle{\bullet}}

\newcommand{\cF}{F_\dt}
\newcommand{\pol}{f}

\newcommand{\Rn}{{\mathbb R}^n}
\newcommand{\An}{{\mathbb A}^n}
\newcommand{\frg}{\mathfrak{g}}
\newcommand{\PW}{{\mathbb P}(W)}

\newcommand{\pos}{{\mathcal Pos}}
\newcommand{\g}{{\gamma}}

\newcommand{\Vaa}{V_0}
\newcommand{\Bp}{B^\prime}
\newcommand{\Bpp}{B^{\prime \prime}}
\newcommand{\bbp}{\mathbf{b}^\prime}
\newcommand{\bbpp}{\mathbf{b}^{\prime \prime}}
\newcommand{\bp}{{b}^\prime}
\newcommand{\bpp}{{b}^{\prime \prime}}
\newcommand{\pb}{\overline{p}}
\newcommand{\Pa}{P \backslash \{a \}}
\newcommand{\Min}{\text{Min}}
\newcommand{\Max}{\text{Max}}
\newcommand{\inc}{\text{inc}}
\newcommand{\tih}{\tilde{h}}

\newcommand{\comment}[1]{{\color{blue} \sf ($\clubsuit$ #1 $\clubsuit$)}}

\def\CC{{\mathbb C}}
\def\GG{{\mathbb G}}
\def\ZZ{{\mathbb Z}}
\def\NN{{\mathbb N}}
\def\RR{{\mathbb R}}
\def\OO{{\mathbb O}}
\def\QQ{{\mathbb Q}}
\def\VV{{\mathbb V}}
\def\PP{{\mathbb P}}
\def\EE{{\mathbb E}}
\def\FF{{\mathbb F}}
\def\AA{{\mathbb A}}

\section*{Introduction}
Letterplace and co-letterplace ideals of a partially ordered set $P$
were introduced and studied in \cite{EHM} and \cite{FGH}.
In the latter paper it was shown that many monomial ideals 
studied in the literature derive from letterplace or co-letterplace
ideals as quotients of these ideals by a regular sequence of variable
differences. These ideals therefore allow a powerful unifying treatment
of many classes of ideals. In this article we: 
\begin{itemize}
\item Develop combinatorial topological results on the homology 
of simplicial  complexes, in particular those associated with bipartite
edge ideals, Section \ref{sec:Top},
\item Use this to compute and give structural results on the multigraded
Betti numbers of resolutions of letterplace ideals, Sections \ref{sec:Betti}
and \ref{sec:Trees}.
\end{itemize}

Given a poset $P$, the $n$'th letterplace ideal $L(n,P)$ is the
monomial ideal generated by monomials
\[ x_{1,p_1}x_{2,p_2} \cdots x_{n,p_n} \]
where $p_1 \leq p_2 \leq \cdots \leq p_n$.
The $n$'th co-letterplace ideal $L(P,n)$ is the monomial ideal generated
by monomials
\[ \Pi_{p \in P} x_{p,i_p}, \]
where $1 \leq i_p \leq n$ and $p < q$ implies $i_p \leq i_q$.

\medskip 
The variables are here $x_{i,p}$ where $(i,p) \in [n] \times P$.
Given a multidegree $R \sus [n] \times P$ we let $\{i \} \times R_i$ be 
the intersection
$R \cap (\{i\} \times P)$, giving subsets $R_1, \ldots, R_n$
of $P$. Each pair $R_i$ and $R_{i+1}$ naturally defines 
a bipartite graph whose edges are $(p_i, p_{i+1}) \in R_i \times R_{i+1}$
where $p_i \leq p_{i+1}$.

\medskip The generators of $L(2,P)$, which give us the graded Betti
numbers in homological degree zero, correspond to pairs of elements
$p_1 \leq p_2$ in the poset $P$. 
The higher Betti numbers of $L(2,P)$ turn out to be 
topological invariants associated with pairs of subsets 
$R_1,R_2$ of the poset $P$ where $R_1 \leq R_2$ for a certain ordering $\leq$. 
More generally the Betti numbers of 
$L(n,P)$ are topological invariants associated with sequences of subsets 
$R_1 \leq R_2 \leq  \cdots \leq R_n$ of the poset $P$.
We show that the Betti numbers of $L(n,P)$ may be
computed very nicely from those of $L(2,P)$. 
We describe in detail the extremal aspects of the
Betti table of $L(n,P)$, like the multigraded Betti numbers
in the first and last linear strands, and in the highest homological 
degree. We also show
that, if $P$ is a union of no more than $c$ chains, then all the
multigraded Betti numbers of $L(n,P)$ may be computed as the homologies of 
simplicial
complexes with $\leq c$ vertices, which can greatly reduce
the task of computing Betti numbers.
When the Hasse diagram of $P$ has
a rooted tree structure we describe completely all multigraded Betti
numbers, and give a simple inductive procedure for 
computing the Betti diagram.

\medskip
To obtain the above we first develop general topological 
results simplifying the computations
of the homology of various simplicial  complexes defined  by i) edge ideals
of bipartite graphs, and ii) more general ideals whose monomials
are given by paths between successive hubs of vertices, in our case
the $R_i$'s.

\medskip
We mention that in the paper \cite{DFN-COLP} we compute resolutions
of co-letterplace ideals. This can be given by a completely explicit
and simple form. On the other hand resolutions of letterplace ideals,
as we investigate here, are much more subtle and rely on the intricate
topological behaviour of $P$.

\medskip
The organization  of the paper is as follows. 
Section 1 defines and gives basic properties of letterplace ideals of posets.
Section 2 develops topological results on the homology of simplicial complexes.
In Section 3 we give the computational and 
structural results on the multigraded
Betti numbers of letterplace ideals, and in Section 4 we
give the recursive procedure for computing Betti diagrams when the
Hasse diagram of $P$ has a tree structure.

\section{Letterplace ideals of posets}
\label{Sec:Pre}
We recall the definitions of the above monomial ideals associated with a poset $P$
and a natural number $n$, and give basic properties
of these ideals which we will use. These ideals were introduced in 
\cite{EHM} and \cite{FGH}.

\medskip 
Let $\kr$ be a field.
If $R$ is a set denote by $\kr[x_R]$ the polynomial ring 
$\kr[x_i]_{i \in R}$, and if $S \sus R$ denote by $m_S$ the squarefree
monomial $\Pi_{i \in S} x_i$.

If $P$ and $Q$ are finite posets, we denote by $\Hom(P,Q)$ the set
of isotone maps $\phi\colon P \pil Q$, i.e. maps such that $p \leq p^\prime$ implies
$\phi(p) \leq \phi(p^\prime)$.  For such a $\phi$ its graph is
\[ \Gamma \phi = \{ (p,\phi(p)) \, | \, p \in P \} \sus P \times Q. \]
Let $L(P,Q)$ be the ideal in $\kr[x_{P \times Q}]$ generated by
the monomials $m_{\Gamma \phi}$ where $\phi \in \Hom(P,Q)$. 

Let $[n] = \{1 < 2 < \cdots < n \}$ be the totally ordered set on
$n$ elements. The ideal $L([n],P)$ in $\kr[x_{[n] \times P}]$ is the
{\it $n$'th letterplace ideal} of $P$ and the ideal $L(P,[n])$ in
$\kr[x_{P \times [n]}]$ is the {\it $n$'th
co-letterplace ideal} of $P$. For short, we write these ideals as $L(n,P)$ and
$L(P,n)$ respectively.

\begin{fact} 
Let $I$ be a homogeneous Cohen-Macaulay ideal of codimension $c$ generated 
in degree $n$ in a polynomial ring. Let $e(I)$ be the multiplicity
of $I$ (i.e. 
the degree of the projective scheme defined by $I$). Then
\begin{equation} \label{eq:LpCoLpMult}
\binom{n+c-1}{c} \leq e(I) \leq n^c. 
\end{equation}
This follows easily by taking an artinian reduction of the ideal $I$.
\end{fact}
 The following properties hold for letterplace ideals:

\medskip
\noindent 1. $L(n,P)$ is a Cohen-Macaulay ideal of codimension equal to 
the cardinality $|P|$ by \cite[Corollary 2.5]{EHM}, see also 
\cite[Corollary 2.4]{FGH}. By \cite[Cor.3.3]{EHM} 
its regularity is $c(n-1) + 1$ where $c$ is the 
maximal cardinality of an antichain in $P$. 

\medskip
\noindent 2. The multiplicity of $L(n,P)$ is the cardinality $|\Hom(P,[n])|$.
This is a consequence of 4. below, since the facets of a simplicial
complex defined by a squarefree monomial ideal correspond one to one to the generators of the Alexander dual ideal.

\medskip
\noindent 3. The upper bound in (\ref{eq:LpCoLpMult}) 
is attained for the letterplace ideal 
$L(n,\underline{c})$ where
$\underline{c}$ is the antichain on $c$ elements, an ideal which is 
a complete intersection.
The lower bound is attained for the letterplace ideal $L(n,[c])$. 
(This letterplace ideal is an initial  ideal of the ideal generated by the
maximal minors of an $n \times (n+c-1)$ matrix
of generic linear forms, see \cite[Section 3]{FGH}.)
Thus as $P$ varies over posets of cardinality $c$, 
the letterplace ideals $L(n,P)$ may be seen as giving Cohen-Macaulay ideals
of codimension $c$ generated in degree $n$ interpolating between these
extreme cases.

\medskip
\noindent 4.
The ideals $L(n,P)$ and $L(P,n)$ are Alexander dual ideals
by \cite[Theorem 1.1]{EHM}, see also \cite[Prop. 1.2]{FGH}.

\medskip
We mention that in \cite{Katt} homological properties
of the ideals $L(P,Q)$ in general are studied, like regularity, projective
dimension, and length of first linear strand.

\section{Topological results}
\label{sec:Top}
We develop topological results for simplicial complexes 
which we will need in the next section when describing the
graded Betti numbers of letterplace ideals.

   The typical situation is derived from the simplicial complex
associated with the edge ideal of a bipartite graph. We may then compute
the homology of the simplicial complex from the homology of an associated
simplicial complex on either of the two vertex sets.

\subsection{Homotopy equivalent simplicial complexes}

The following basic situation will be useful for us.
Let $X$ be a simplicial complex with vertex set $V$.
Let $a_1$ and $a_2$ be vertices in $V$ and suppose that when
$\{ a_1 \} \cup G$ and $\{ a_2 \} \cup G$ are in $X$, then
$\{ a_1, a_2 \} \cup G$ is in $X$. Let $\Vaa = V \backslash \{a_1, a_2 \}$,
and let $Y$ be the simplicial complex on
$W = \Vaa \cup \{a \}$ such that
\begin{itemize} 
\item The restrictions 
$X_{|\Vaa} = Y_{|\Vaa}$.
\item For $G \sus V_0$, then $G \cup \{a \}$
is a face of $Y$ iff either $G \cup \{ a_1 \}$ or $G \cup \{ a_2 \}$ is
a face of $X$. 
\end{itemize}


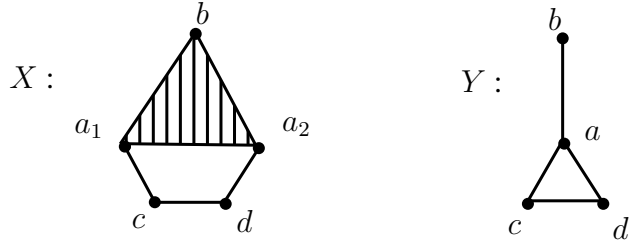
\scalebox{1} 
{
\begin{pspicture}(-1,-1.6129688)(7.502812,1.6129688)
\psline[linewidth=0.04,fillstyle=vlines,hatchwidth=0.04,hatchangle=0.0](1.5409375,-0.32546875)(2.5409374,1.1545312)(3.3409376,-0.34546876)(1.5809375,-0.32546875)(1.6009375,-0.34546876)
\psline[linewidth=0.04](1.6009375,-0.36546874)(2.0009375,-1.1054688)(2.9209375,-1.1054688)(3.3809376,-0.38546875)(3.4009376,-0.40546876)
\psdots[dotsize=0.12](1.6209375,-0.36546874)
\psdots[dotsize=0.16](1.6009375,-0.36546874)
\psdots[dotsize=0.16](3.3809376,-0.38546875)
\psdots[dotsize=0.16](2.9409375,-1.1254687)
\psdots[dotsize=0.16](2.0009375,-1.1054688)
\psdots[dotsize=0.16](2.5409374,1.1345313)
\psline[linewidth=0.04](7.4209375,1.1145313)(7.4209375,-0.28546876)(6.9609375,-1.0854688)(7.9409375,-1.0854688)(7.4409375,-0.30546874)(7.4409375,-0.32546875)
\psdots[dotsize=0.16](7.4209375,1.0745312)
\psdots[dotsize=0.16](7.4409375,-0.32546875)
\psdots[dotsize=0.16](6.9609375,-1.1254687)
\psdots[dotsize=0.16](7.9609375,-1.1254687)
\usefont{T1}{ptm}{m}{n}
\rput(1.1223438,-0.12046875){$a_1$}
\usefont{T1}{ptm}{m}{n}
\rput(3.8823438,-0.10046875){$a_2$}
\usefont{T1}{ptm}{m}{n}
\rput(2.6323438,1.4195312){$b$}
\usefont{T1}{ptm}{m}{n}
\rput(1.7823437,-1.3604687){$c$}
\usefont{T1}{ptm}{m}{n}
\rput(3.1923437,-1.3604687){$d$}
\usefont{T1}{ptm}{m}{n}
\rput(6.782344,-1.4404688){$c$}
\usefont{T1}{ptm}{m}{n}
\rput(8.192344,-1.4404688){$d$}
\usefont{T1}{ptm}{m}{n}
\rput(7.3123436,1.3395313){$b$}
\usefont{T1}{ptm}{m}{n}
\rput(7.822344,-0.18046875){$a$}
\usefont{T1}{ptm}{m}{n}
\rput(0.3665625,0.5595313){$X$ :}
\usefont{T1}{ptm}{m}{n}
\rput(6.3465624,0.49953124){$Y$ :}
\end{pspicture} 
}

\begin{proposition} \label{TopProHtp}
$X$ and $Y$ are homotopy equivalent.
\end{proposition}

\begin{proof} 
A short argument using the category of CW-complexes is the following:
let $A_2$ be the subcomplex of $X$ generated by the faces containing
$\{ a_1, a_2 \}$, and let $A_1$ be the subcomplex of $Y$
generated by faces $G$ such that $(G \backslash \{a \}) \cup \{a_1, a_2 \}$
is a face of $X$. The quotient CW-complexes
$X/A_2$ and $Y/A_1$ can then be seen to be the same. By
Propositions 0.16 and 0.17 in \cite{Hat}, $X$ is homotopic to 
$X/A_2$, and $Y$ is homotopic to $Y/A_1$, and therefore $X$ and $Y$
are homotopic.

\medskip
A detailed argument using simplicial complexes is as follows:
the map $ f: V \pil V_0 \cup \{a\}$ sending $a_1$ and $a_2$
to $a$ induces a 
simplicial map $\tilde{f}: X \pil Y$, sending a face $F$ of $X$ to a face $f(F)$
of $Y$. By Lemma 2.6 of \cite{Lov} we need to check that for
every face $G$ of $Y$ the subcomplex of $X$ induced on $f^{-1}(G)$
is contractible. This will imply that $X$ and $Y$ are homotopy equivalent.

If $G$  does not contain $a$, then $f^{-1}(G) = G$ 
is also a face of $X$ (hence a simplex, hence contractible).
If $G = G^\prime \cup \{a\}$, 
then at least one of $G^\prime \cup \{a_1\}$  and $G^\prime \cup \{a_2\}$ 
is in $X$. 
The inverse image of $G$ by $f$ is $G^\prime \cup \{a_1, a_2\}$.
If both $G^\prime \cup \{a_1\}$ and $G^\prime \cup \{a_2\}$ lie in $X$, 
then also $G^\prime \cup \{a_1, a_2\}$ is in $X$ and we are done.

If $G^\prime \cup \{a_1\}$ is in $X$ and $G^\prime \cup \{a_2\}$ is not, 
then we claim that the subcomplex $C$ of $X$ induced on 
$G^\prime \cup \{a_1, a_2\}$ is a cone on $a_1$ and hence is contractible, 
yielding the result. Let $F^\prime \sus G^\prime$. It is a face of $C$ not
containing $a_1$. Since $G^\prime \cup \{ a_1 \}$ is in $C$,
$F^\prime \cup \{a_1\}$ is in $C$. If also $F^\prime \cup \{ a_2 \}$ is in 
$C$, then $F^\prime \cup \{ a_1, a_2 \}$ is in $X$ and so in $C$.
The upshot is that any facet of $C$ will contain $a_1$, and so 
$C$ is a cone.
\end{proof}

\subsection{Simplicial complexes from bipartite graphs}
\label{Subsec:Top:Bip}

We consider a bipartite graph with vertex set $A \cup B$ and
edges between $A$ and $B$. Let $I_X$ be the edge ideal of this
graph and $X$ the simplicial complex with Stanley-Reisner ideal $I_X$.

\begin{lemma}
If some vertex in $A$ or $B$ is not incident to an edge, then 
$X$ is contractible.
\end{lemma}

\begin{proof}
This is clear since in this case $X$ is a cone over this vertex.
\end{proof}


\begin{lemma} \label{TopLemaap}
Let $a$ and $a^\prime$ in $A$ be such that the
set of neighbours of $a^\prime$ contains the set of neighbours of $a$. 
Let $Y$ be the induced simplicial complex on $(A\backslash \{ a^\prime\})
\cup B$. Then $X$ and $Y$ are homotopy equivalent.
\end{lemma}

\begin{proof}
Note that if a face $F$ of $X$ contains $a^{\prime}$ then $F \cup \{ a \}$
is also a face of $X$. Then $Y$ is obtained from $X$ as in 
Proposition \ref{TopProHtp}. Thus $Y$ and $X$ are homotopy equivalent.
\end{proof}

In the following, subsets of a set denoted by a capital letter will
often be denoted by bold lower case letters.
Let $Y$ be the simplicial complex on $B$ consisting of 
the subsets $\bb$ of $B$ such that $\bb \cup \{ a \}$ is a face of 
$X$, for some $a \in A$. That is, there is no edge between $a$ and 
any element of $\bb$.

  The following is very close to \cite[Thm. 4.7]{Dal}. We also thank
Morten Brun for valuable ideas concerning its proof.

\begin{proposition} \label{TopProSusp}
$X$ is homotopy equivalent to the suspension of $Y$. 
\end{proposition}

\begin{proof}
By Lemma 10.4 in \cite{Bj} one has that, if $X$ can be written as the 
union of two contractible
subcomplexes $X_0$ and $X_1$, then $X$ is homotopy equivalent to the suspension
of $X_0 \cap X_1$.
In our case, let \[X_0 = \{\aa \cup \bb \in X \mid \bb \in Y\}\] and 
\[X_1 = \{\aa \cup \bb \in X \mid \aa = \emptyset\}\] 
Note that $X_1$ identifies as the simplex on $B$.
We see that $X_0 \cup X_1 = X$ and $X_0 \cap X_1 = Y$. 
Moreover $X_1$, being a simplex, is clearly contractible. It is left to 
show that $X_0$ is contractible as well.
This comes from Lemma 4.2 in \cite{Dal} by taking $\Gamma = Y$ and 
$\Gamma_j = \{\bb \in B \mid \bb \cup \{ a_j \} \in X\}$.
\end{proof}

\subsection{The join} \label{Subsec:join}
Recall that if $X$ is a simplicial complex on a vertex set $V$, and
$Y$ is a simplicial complex on a vertex set $W$, then the
join $X * Y$ is the simplicial complex on $V \cup W$ whose faces
are all $F \cup G$ where $F$ is a face of $X$ and $G$ is a face of $Y$.
If $X$ is homotopy equivalent to a simplicial complex $X^\prime$
on $V^\prime$, then $X * Y$ and $X^\prime * Y$ are homotopy equivalent.

We define the polynomial of reduced cohomology as
\[ \tH(X,t) = \sum_{i \geq -1} t^i \dim_\kr \tH^i(X,\kr). \]

\begin{proposition} \label{pro:TopJoin}
Let $X$ and $Y$ be simplicial complexes on disjoint vertex sets.
Then the polynomial
\[ t \tH(X*Y,t) = t \tH(X,t) \cdot t \tH(Y,t). \]
Consequently if $X_1, \ldots, X_n$ are simplicial complexes on disjoint vertex sets, then
\[ t \tH(X_1 * \cdots * X_n) = \prod_{i = 1}^n t \tH(X_i,t). \]
\end{proposition}

\begin{proof} Let $\tC(X;\kr)$ be the augmented chain complex of 
the simplicial complex $X$, see \cite[Sec.1.3]{StMi}.
(Its homology $H_q(\tC(X;\kr))$ is the reduced homology $\tH_q(X;\kr)$.)
Then the augmented chain complex of the join $X * Y$ is the tensor
product of the augmented chain complexes of $X$ and $Y$, but with
a shift in homological degree:
\[ \tC(X*Y;\kr)[-1] \iso \tC(X;\kr) \te_\kr \tC(Y;\kr). \]

By \cite[Thm. 3B.5]{Hat} the reduced homology of the join
\[ \tH_{q+1}(X*Y; \kr) \iso \bigoplus_{i} 
\tH_i(X, \kr) \otimes \tH_{q-i}(Y, \kr). \]
This proves the statement.



\end{proof}

\subsection{Simplicial complexes homotopic to joins}

We consider $X_{AB}$ a simplicial complex on vertex set $A \cup B$
with Stanley-Reisner ideal $I_{AB}$, 
and $X_{BC}$ a simplicial complex on vertex set $B \cup C$
with Stanley-Reisner ideal $I_{BC}$. 
We suppose no generator of $I_{AB}$ is divisible by a quadratic
monomial generator $x_{b_1}x_{b_2}$ where $b_1,b_2$ is in $B$
and similarly for $I_{BC}$.

Let $X$ be the simplicial complex on $A \cup B \cup C$ 
whose Stanley-Reisner ideal $I$ is generated by all $x_{\aa} \in I_{AB}$,
$x_{\cc} \in I_{BC}$ (where $\aa \sus A$ and $\cc \sus C$),
 and $x_{\aa} x_{b_i} x_{\cc}$ where 
$x_{\aa}x_{b_i} \in I_{AB}$ and $x_{b_i}x_{\cc} \in I_{BC}$. 
Then $X$ consists
of all $\aa \cup \bb \cup \cc$ such that for each $b_i \in \bb$
then either $\aa \cup \{b_i \}$ is in $X_{AB}$ or $\{b_i \} \cup \cc$ 
is in $X_{BC}$. 

Let $\Bp$ and $\Bpp$ be copies of $B$. We thus get a simplicial
complex $X_{A\Bp}$ on $A \cup \Bp$, a copy of $X_{AB}$, and so on.

\begin{theorem}
The join $X_{A\Bp} * X_{\Bpp C}$ is homotopy equivalent to $X$.
\end{theorem}

\begin{proof}
Let $\Bp = \{ b_1^\prime, \ldots, b_p^\prime \}$ and $\Bpp = \{ b_1^{\prime \prime}, 
\ldots, b_p^{\prime \prime} \}$. We will apply Proposition \ref{TopProHtp}
by making $b_i^\prime$ and $b_i^{\prime \prime}$ successively equal
for $i = 1, \ldots, p$, and after identification we rename it as 
$b_i$.

Let $X_i$ be the simplicial complex on the vertex set 
\[ A \cup B_{\leq i} \cup \Bp_{> i} \cup \Bpp_{> i} \cup C \]
consisting of all subsets 
\begin{equation} \label{eq:Topabb}
\aa \cup \bb_{\leq i} \cup \bbp_{> i}
\cup \bbpp_{> i} \cup \cc
\end{equation} such that: 
\begin{enumerate} \label{TopEnumInX}
\item[i.] $\aa \cup \bbp_{> i}$ is in $X_{A\Bp}$.
\item[ii.] $\bbpp_{> i} \cup \cc$ is in $X_{\Bpp C}$.
\item[iii.] $\aa \cup \bb_{\leq i} \cup \cc$ is in $X$.
\end{enumerate}

Note that $X_0$ is the join $X_{A\Bp} * X_{\Bpp C}$ and $X_p = X$. 
We will show that $X_i$ and $X_{i+1}$ are homotopy equivalent by applying Proposition \ref{TopProHtp}.
Let 
\[ V_0 = (A \cup B_{\leq i} \cup \Bp_{> i} \cup \Bpp_{> i} \cup C) 
\backslash \{ \bp_{i+1}, \bpp_{i+1} \}. \]
Clearly $X_i$ and $X_{i+1}$ have the same restriction to $V_0$. Then let
$R$ of the form \eqref{eq:Topabb} be a subset of $V_0$. We must show that:
\begin{enumerate}
\item If $R \cup \{ \bp_{i+1} \}$ and $R \cup \{ \bpp_{i+1} \}$ are in
$X_i$, then $R \cup \{ \bp_{i+1}, \bpp_{i+1} \}$ is in $X_i$.
\item If  $R \cup \{ \bp_{i+1} \}$ or $R \cup \{ \bpp_{i+1} \}$ is in
$X_i$, then $R \cup \{ b_{i+1} \}$ is in $X_{i+1}$.
\item If  $R \cup \{ b_{i+1} \}$ is in $X_{i+1}$, then  
$R \cup \{ \bp_{i+1} \}$ or $R \cup \{ \bpp_{i+1} \}$ is in $X_i$.
\end{enumerate}

Part (1) is clear from the criteria above, i., ii. and iii.
For part (2), if $R \cup \{b_{i+1} \}$ is not in $X_{i+1}$, then
either i. $\aa \cup \bbp_{> i+1}$ is not in $X_{A\Bp}$
and so none of $R \cup \{\bp_{i+1} \}$ or $R \cup \{ \bpp_{i+1} \}$
is in $X_i$, or ii. $\bbpp_{> i+1} \cup \cc$ is not in $X_{\Bpp C}$
and so none of $R \cup \{\bp_{i+1} \}$ or $R \cup \{ \bpp_{i+1} \}$
is in $X_i$,
or iii. $\aa \cup \bb_{\leq i} \cup \{ b_{i+1} \} \cup \cc$ is not in 
$X$. Then, say, $\aa \cup \{ b_{i+1} \}$ is not in $X_{AB}$, and so
$R \cup \{ \bp_{i+1} \}$ is not in $X_i$. Similarly we get
$R \cup \{ \bpp_{i+1} \}$ is not in $X_i$. But this contradicts our
assumption.

Part (3): If $R \cup \{b_{i+1} \}$ is in $X_{i+1}$, then
$\aa \cup \bb_{\leq i+1} \cup \cc$ is in $X$. Either $\aa \cup \{b_{i+1} \}$
is in $X_{AB}$ or $\{ b_{i+1} \} \cup \cc$ is in $X_{BC}$. 
In the first case we see that 
\begin{enumerate}
\item $\aa \cup \bbp_{> i+1} \cup \{b^\prime_{i+1}\}$ is in $X_{A \Bp}$.
This is because $\aa \cup \bbp_{> i+1} \in X_{A \Bp}$ and
$\aa \cup \{ b_{i+1} \}$ is in $X_{A B}$.
\item $\bbpp_{> i+1} \cup \cc = \bbpp_{> i} \cup \cc$ is in $X_{\Bpp C}$.
\item $\aa \cup \bb_{\leq i} \cup \cc$ is in $X$.
\end{enumerate}

Thus $R \cup \{ \bp_{i+1} \}$ is in
$X_i$. 

  In the other case, when $\{ b_{i+1} \} \cup \cc$ is in $X_{BC}$,
we find that $R \cup \{ \bpp_{i+1} \}$ is in $X_i$.
\end{proof}

Now consider a graph whose set of vertices is a disjoint
union $A_0 \cup \cdots \cup A_n$, and such that each edge in this graph is
between $A_{i-1}$ and $A_i$ for some $i$. Let $X_i$ be the simplicial
complex on $A_{i-1} \cup A_i$ with Stanley-Reisner ideal the
edge ideal given by the edges between $A_{i-1}$ and $A_i$.
Let $X$ be the simplicial complex whose Stanley-Reisner ideal is
generated by products $x_{a_0} x_{a_1} \cdots x_{a_n}$
where $a_i \in A_i$ and $\{a_{i-1}, a_i \}$ are edges in the graph.

\begin{corollary} \label{Cor:Top:Htpjoin}
The simplicial complex $X$ is homotopy equivalent to the 
join $X_1 * X_2 * \cdots * X_n$. 
\end{corollary}

\begin{proof} 
Let $X_{\leq i}$ be the simplicial complex whose monomials
correspond to the edge paths from $A_0$ to $A_i$. 
Then $X_{\leq 1} = X_1$ and $X_{\leq n} = X$. We argue by induction. 
By the previous theorem $X_{\leq i} * X_{i+1}$ is homotopy 
equivalent to $X_{\leq i+1}$. By induction $X_{\leq i}$
is homotopy equivalent to $X_1 * X_2 * \cdots * X_i$ and
so $X_{\leq i+1}$ is homotopy equivalent to $X_1 * \cdots * X_{i+1}$. 
\end{proof}

\section{Resolutions of letterplace ideals}
\label{sec:Betti}
 In Section \ref{Sec:Pre} we introduced the $n$'th letterplace ideal
$L(n,P)$ of a poset $P$. In this section we investigate the 
resolution  of $L(n,P)$. By all likelihood it is not possible to give 
this in explicit form
since a resolution of $L(n,P)$ in general seems to be as complicated as a 
resolution  of a Stanley-Reisner ring. 
For instance the resolution of 
$L(n,P)$ may depend on the characteristic, see Remark \ref{chardep} below. 

  Nevertheless we shall show how to calculate the multigraded Betti 
numbers of $L(n,P)$ in a way that substantially reduces the computation
compared to Hochster's formula. If $P$ is a union of no more than
$c$ chains, then every multigraded Betti number can be computed from
a simplicial complex on no more than $c$ vertices. For the class
of posets whose Hasse diagram is a rooted tree, we show there
is a simple inductive procedure for computing the graded Betti numbers.

\subsection{Betti numbers of simplicial complexes}
We first recall Hochster's formula for the  multigraded Betti numbers
of a simplicial complex. So let $X$ be a simplicial complex on the 
vertex set $V$. Let $I_X$ be the Stanley-Reisner ideal of $X$ in 
$\kr[x_v; v \in V]$. If $R \sus V$ we let $X_{|R}$ be the restricted 
simplicial complex consisting of all $F \in X$ with $F \sus R$. 

Let 
\[ I_X \vpil F_0 \vpil F_1 \vpil \cdots  \]
be the minimal free resolution of $I_X$. Since $I_X$ is a squarefree 
ideal, we can write each 
\[ F_i = \bigoplus_{R \sus V} S(-R)^{\beta_{i,R}}. \]
The number $\beta_{i,R} = \beta_{i,R}(I_X)$  is the multigraded Betti
number of $I_X$ of degree $R$ and homological degree $i$. 

\begin{theorem}[Hochster] The Betti number $\beta_{i,R}(I_X)$ is
the dimension as a $\kr$-vector space of the reduced cohomology group
\[ \tH^{|R|-i-2}(X_{|R}, \kr).\]
\end{theorem}

We denote by $\Delta(n,P)$ the simplicial complex corresponding
to the Stanley-Reisner ideal $L(n,P)$. Its set of vertices is 
$[n] \times P$. It follows from the theorem above that for $R \sus
[n] \times P$ the Betti number $\beta_{i,R}(L(n,P))$ is the dimension of
\[ \tH^{|R| -i -2}(\Delta(n,P)_{|R}, \kr). \]

The following is noteworthy although we do not use it here.
\begin{fact} The simplicial complex $\Delta(n,P)$ is a ball
of codimension $|P|$ in the simplex on the vertex set $P \times [n]$,
\cite[Thm. 5.1]{DFN-COLP} (with one exception: when $P$ is an antichain). 
The boundary of $\Delta(n,P)$ is thus a simplicial sphere, 
see \cite[Sec. 5]{DFN-COLP} for more on this.
\end{fact}

\subsection{Multigraded Betti numbers of $L(2,P)$}

For $R \sus [2] \times P$ write $R = \{ 1 \} \times R_1 \cup
\{ 2 \} \times R_2$. Since both $R_1$ and $R_2$ are subsets of $P$, they
each get an induced poset structure.
Let $\min(R_2)$ be the minimal elements of $R_2$ and $\max(R_1)$ 
the maximal elements of $R_1$. We then get a bipartite graph on 
$\max(R_1) \cup \min(R_2)$ with edges $\{p_1,p_2\}$ (where $p_1 \in \max(R_1)$
and $p_2 \in \min(R_2)$) if $ p_1 \leq p_2$ in the poset $P$.
    
Let $X(R)$ be the simplicial complex on $\max(R_1) \cup \min(R_2)$
whose Stanley-Reisner ideal is the edge ideal of this bipartite graph.
Let $Y_2(R)$ be the simplicial complex on $\min(R_2)$ 
consisting of all $B \sus \min(R_2)$ such that $\{a \} \cup B$
is in $X(R)$ for some $a$ in $\max(R_1)$, or alternatively there is some
$a$ in $\max(R_1)$ with no $b$ in $B$ dominating $a$.
This is the simplicial complex constructed as in
Subsection \ref{Subsec:Top:Bip}.
Further let $Y_1(R)$ be the simplicial complex on $\max(R_1)$
consisting of all
subsets $A \sus \max(R_1)$ such that $A \cup \{b\}$ is in $X(R)$ for
some $b \in \min(R_2)$, i.e. no $a \leq b$ for $a \in A$. 


\begin{theorem}
The multigraded Betti number $\beta_{i,R}(L(2,P))$ is the dimension of 
\[ \tH^{|R|-i-2}(X(R),\kr) \iso \tH^{|R|-i-3}(Y_1(R),\kr) \iso
\tH^{|R|-i-3}(Y_2(R), \kr). \]
\end{theorem}

\begin{proof}
By Hochster's formula $\beta_{i,R}(L(2,P))$ is
\[ \dim_{\kr} \tH^{|R| -i-2}(\Delta(2,P)_{|R},\kr).\]
The Stanley-Reisner ring of $\Delta(2,P)_{|R}$ is the edge ideal
of the bipartite graph on $R_1 \cup R_2$ whose edges are $\{r_1,r_2\}$
where $r_1 \leq r_2$. By Lemma \ref{TopLemaap}, $\Del(2,P)_{|R}$ is 
homotopy equivalent to $X(R)$, and by Proposition \ref{TopProSusp}
$X(R)$ is homotopy equivalent to the suspension of $Y_1(R)$, and also
of $Y_2(R)$ from which the above follows.
\end{proof}

The following shows that although $P$ may be large, there is a 
uniform bound on the size of the simplicial complexes one uses to compute
the multigraded Betti numbers.

\begin{corollary}
If the poset $P$ is a union of no more than $c$ chains, the multigraded
Betti numbers of $L(2,P)$ 
can be computed as the homology of simplicial complexes
with no more than $c$ vertices.
\end{corollary}

\begin{proof} If $P$ is the union of no more than $c$ chains,
then every antichain in $P$ has cardinality
less than or equal to $c$. 
The simplicial complexes $Y_1(R)$ and $Y_2(R)$ are then each complexes
on no more than $c$ vertices, since $\max(R_1)$ and $\min(R_2)$ are
antichains.
\end{proof}

Given a simplicial complex $Z$, denote by $\Sigma^k{Z}$ its $k$'th iterated suspension. We now introduce another reduction that helps us compute $X(R)$ by considering a smaller simplicial complex.

\begin{proposition} \label{iterated_suspension}
The complex $X(R)$ equals $\Sigma^{|\max(R_1) \cap \min(R_2)|}C$, where $C$ is the subcomplex of $X(R)$ induced by $(\max(R_1) \cup \min(R_2)) \setminus (\max(R_1) \cap \min(R_2))$. As a consequence, 
\[\tH^i(X(R), \kr) \cong \tH^{i-|\max(R_1) \cap \min(R_2)|}(C, \kr)\] for all $i$.
\end{proposition}
\begin{proof}
Given any $b \in \max(R_1) \cap \min(R_2)$, one has that $b$ is comparable only to itself inside $\max(R_1) \cup \min(R_2)$. For this reason, considering the Stanley-Reisner ideals of $X(R)$ and $C$, one has that \[I_{X(R)} = I_C + (x_{1, b}x_{2, b} \mid b \in \max(R_1) \cap \min(R_2)).\]
The rest of the claim is a general fact about suspensions: for any simplicial complex $Z$ one has that $\tH^{i+k}(\Sigma^k{Z}, \kr) \cong \tH^i(Z, \kr)$ for all $i$.
\end{proof}

We record the following corollary for use in the proof of Proposition \ref{pro:MultiHighhom}.

\begin{corollary} \label{lem:MultiIntersect}
If $\max(R_1) \neq \min(R_2)$
then $\tH^p(Y_i(R), \kr) = 0$ for $p \leq |\max(R_1)  \cap \min(R_2)| - 2$.
If $\max(R_1) = \min(R_2)$ then 
\[ \tH^p(Y_i(R),\kr) = \begin{cases} \kr & p = |\max(R_1)| - 2 \\
                               0 & p \neq |\max(R_1)| - 2
                   \end{cases}
\]
\end{corollary}

\begin{remark} \label{chardep}
Betti numbers of $L(2, P)$ are characteristic-dependent in general.
In fact, one can ``simulate the behaviour of any simplicial complex'' in this context.
The construction, essentially found in \cite{Dal}, goes as follows: assume $\Delta$ is a simplicial complex on $[n]$ with facets $F_1, \ldots, F_m$ 
and take the bipartite graph $G$ with bipartition 
given by $A \cup B = \{a_1, \ldots, a_n\} \cup \{b_1, \ldots, b_m\}$ and such that $\{a_i, b_j\}$ is an edge of $G$ 
precisely when $i$ does not belong to $F_j$ in $\Delta$. Consider then the poset $P$ whose nontrivial 
covering relations are precisely those of the form $a_i < b_j$ where $\{a_i, b_j\}$ is an edge of $G$.
Now, if $R = \{1\} \times A \cup \{2\} \times B$, one checks that $\Delta(2, P)_{|R}$ is the simplicial complex 
whose Stanley-Reisner ideal is the edge ideal of $G$. By Proposition \ref{TopProSusp} and by construction, 
this complex turns out to be homotopy equivalent to the suspension 
of the original complex $\Delta$. Since Hochster's formula 
holds, to get an example where Betti numbers of $L(2, P)$ are indeed characteristic-dependent 
it now suffices to consider as our $\Delta$ the usual triangulation of the real projective plane.
\end{remark}

We now define a transitive order on nonempty subsets of the poset.
\begin{definition} Let $A$ and $B$ be nonempty subsets of $P$.
We write $A \leq B$ if:
\begin{itemize}
\item For every maximal element $a$ in $A$ there is a minimal element 
$b \in B$ such that $a \leq b$. 
\item For every minimal element $b$ in $B$ there is a maximal element $
a \in A$ such that $a \leq b$. 
\end{itemize} 
Note that this is not a partial order since normally not $A \leq A$
(so it is not reflexive). In fact, $A \leq A$ iff $A$ is an antichain.
\end{definition}

The generators of the ideal $L(2,P)$ correspond to pairs
$p_1 \leq p_2$ in the poset $P$. These generators give
the multigraded Betti numbers of $L(2,P)$ in homological degree zero.
The following provides an analogue for the higher multigraded Betti numbers,
whose values will depend on the topology
of the poset relative to the multidegrees.

\begin{corollary} \label{cor:BettiL2P}
If not $R_1 \leq R_2$ then  $\Del(2,P)_{|R}$ is
contractible, and so all the reduced homology of $\Del(2,P)_{|R}$ vanishes.
Hence $\beta_{i,R}$ can be nonzero only if $R_1 \leq R_2$.
\end{corollary}

\begin{proof} If some $p_1 \in \max(R_1)$ is not dominated by any $p_2 \in 
\min(R_2)$,
the simplicial complex $Y_2(R)$ 
is the full simplex on $\min(R_2)$, 
since no $q \in \min(R_2)$ is such that $p_1 \leq q$.
Analogously if some $p_2 \in \min(R_2)$ does not dominate any $p_1 \in 
\max(R_1)$, then $Y_1(R)$ is the full simplex on $\max(R_1)$.
\end{proof}

\subsection{Multigraded Betti numbers of $L(n,P)$}
For $R \sus [n] \times P$ write 
$R = \cup_{i=1}^n \{i \} \times R_i$.
Each $R_i$ gets an induced poset structure from $P$. We get a graph
on $R$ whose edges are between $\{i\} \times R_i$ and $\{i+1\} \times R_{i+1}$
for $i = 1, \ldots, n-1$. The edges are  pairs $\{ (i,p_i), (i+1,p_{i+1}) \}$
with $p_i \leq p_{i+1}$.
The Stanley-Reisner ideal of $\Delta(n,P)_{|R}$ is generated by 
products $x_{1,p_1}x_{2,p_2} \ldots x_{n,p_n}$ where pairs of successive indices
are edges of the graph.  Let $\Delta_i(n,P)_{|R}$ be the simplicial complex
whose Stanley-Reisner ideal is the edge ideal of the bipartite graph
on $\{i\} \times R_i \cup \{ i+1 \} \times R_{i+1}$. 

Let $X_i(R)$ be the simplicial  complex which is the restriction of 
$\Delta_i(n,P)_{|R}$ to $\max(R_i) \cup \min(R_{i+1})$.

\begin{proposition} Let $R \sus [n] \times P$. The restriction
$\Delta(n,P)_{|R}$ is homotopy equivalent to the join
\[ X_1(R) * X_2(R) * \cdots * X_{n-1}(R). \]
\end{proposition}

\begin{proof}
This follows by Corollary \ref{Cor:Top:Htpjoin}.
\end{proof}

Recall the polynomial $\tH(X_i(R),t)$ defined in Subsection \ref{Subsec:join}.
Let $r$ be the cardinality $|R|$ and $\beta(R,t)$ 
the Betti polynomial 
\[ t^r\beta_{0,R} + t^{r-1} \beta_{1,R} + \cdots + t \beta_{r-1,R} \]
where $\beta_{i,R} = \beta_{i,R}(L(n,P))$.
We say that a Betti number $\beta_{i,R}$ is in the {\it $(|R|-i)$-linear strand}
of $L(n,P)$. We thus see that the coefficient of $t^p$ in $\beta(R,t)$ 
above is the Betti number of multigrade $R$ in the $p$-linear
strand of $L(n,P)$. 

\begin{theorem} \label{The:LPR:BettiLnP}
The Betti polynomial $\beta(R)$ of $L(n,P)$ is 
\[ \beta(R,t) = t^n \prod_{i=1}^{n-1}\tH(X_i(R),t). \]
\end{theorem}

\begin{proof} 
By the previous proposition and Proposition \ref{pro:TopJoin}
\[ t \tH(\Delta(n,P)_{|R},t) = \prod_{i = 1}^{n-1} t \tH(X_i(R),t). \]
By Hochster's formula
\[ \beta(R,t) = t^2 \tH( \Delta(n,P)_{|R}, t) \]
and so the statement follows.
\end{proof}

Note that $X_i(R)$ is homotopic to the suspension of the complex 
$Y_{i,\max}(R)$ on $\max(R_i)$ constructed as in Subsection 
\ref{Subsec:Top:Bip}.
It is similarly homotopic to the suspension $Y_{i+1,\min}(R)$ on $\min(R_{i+1})$
also constructed there.

\begin{corollary} 
The Betti number $\beta_{i,R}(L(n,P))$ can be nonzero only if 
$R_1 \leq R_2 \leq \cdots \leq R_n$. 
\end{corollary}

\begin{proof}
This follows from Corollary \ref{cor:BettiL2P} because each 
$\tH(X_i(R),t)$ must be nonvanishing.
\end{proof}

\subsection{Extremal properties of the Betti table of $L(n,P)$}

For the letterplace ideal $L(n,P)$ we now
characterize i) the Betti numbers in the first
and last linear strands, ii) the Betti numbers in the
highest homological degree $|P|-1$, and iii) starting homological
degrees of the linear strands. As a consequence we answer
a problem posed in \cite[Sec.7]{Katt} on when the ideal $L(n,P)$ is
a Cohen-Macaulay level ideal.

The ideal $L(n,P)$ is 
generated by the  monomials $x_{1,p_1}x_{2,p_2}\cdots x_{n,p_n}$
with $p_1 \leq p_2 \leq \cdots \leq p_n$ and so the first linear
strand is the $n$-linear strand. The following characterizes
the Betti numbers in this strand.

\begin{proposition}
The multidegree $R$ occurs in the first $n$-linear strand 
in the resolution of $L(n,P)$ iff for every $i = 1, \ldots, n-1$ 
and every pair of elements 
$p \in R_i$ and $q \in  R_{i+1}$ we have $p \leq q$. 
Then the Betti number $\beta_{i,R}(L(n,P)) = 1$.
\end{proposition}

\begin{proof}
The degree $R$ occurs in the first linear strand iff $|R| = i+n$ and 
$\beta_{i,R}$ is nonzero. Thus $t^n$ occurs in $\beta(R,t)$ with 
nonzero coefficient $\beta_{i,R}$. But then by Theorem 
\ref{The:LPR:BettiLnP} each $\tH(X_i(R),t)$ has nonzero constant term,
i.e. $\tH^0(X_i(R),\kr)$ is nonzero. Since this equals
$\tH^{-1}(Y_{i,\max}(R),\kr)$ this last simplicial complex is the
complex $\{ \emptyset \}$ and so the induced bipartite graph on 
$\max(R_1) \cup \min(R_2)$ is the complete bipartite graph.
\end{proof}

\begin{example}
Let $P = [m]$ be the totally ordered poset on $m$ elements. 
For two sets $A,B \sus [m]$ we now have $A \leq B$ if $a \leq b$
for every $a \in A$ and $b \in B$. Subsets 
$R_1 \leq R_2 \leq \cdots \leq R_n$ of $[m]$ give a multidegree
$R = \cup_i \{i \} \times R_i$ such that $\beta_{|R|-n,R}(L([n],[m])) = 1$.
Moreover every nonzero Betti number of $L([n],[m])$ arises in 
this  manner.
\end{example}

By \cite[Cor. 3.3]{EHM} the regularity of $L(n,P)$ is
$c(n-1) + 1$ where $c$ is the size of the largest antichain in $P$.  
We characterize the multidegrees in this last linear strand.

\begin{proposition}
The multidegree $R$ occurs in the last $(c(n-1) +1)$-linear strand if
for each $i$ one has that $R_i \leq R_{i+1}$, the sets $\max(R_i)$ and $\min(R_{i+1})$ have the
same cardinality $c$ and 
there is a one to one correspondence between these sets
such that the corresponding
pairs of elements are the only comparability relations between these sets.
The Betti number in this case is $\beta_{i,R}(L(n,P)) = 1$.
(Example case: let $R_1 = R_2 = \cdots = R_n$ 
be an antichain of cardinality $c$.) 
\end{proposition}

\begin{proof}
The degree $R$ occurs in the $(c(n-1) +1)$-linear strand iff 
 $t^{c(n-1) +1}$ occurs in $\beta(R,t)$ with 
nonzero coefficient. Dividing this polynomial by $t^n$,
the power $c(n-1) +1-n = (c-1)(n-1)$ must occur with nonzero coefficient
in $\prod_{i = 1}^{n-1} \tH(X_i(R),t)$. But $\tH^j(X_i(R);\kr) 
\iso \tH^{j-1}(Y_{i,max}(R);\kr)$
vanishes for $j-1 \geq c-1$. Thus these cohomology groups must all 
be nonvanishing for $j = c-1$. Then each $Y_{i,max}(R)$ must be the boundary of 
the $(c-1)$-simplex and the result follows.
\end{proof}

\begin{proposition} Let $i = 1,\ldots, c-1$.
 For $j > in - (i-1)$ the $j$'th linear
strand is zero in homological degrees $\leq i-1$.
Furthermore the $(i+1)n - i$ linear strand starts in homological 
degree $i$. 
\end{proposition}

\begin{proof} 
Let $A$ be an antichain with $i+1$ elements. Let 
$R_1 = R_2 = \cdots = R_n = A$. Then each $Y_{i,max}(R)$ is the boundary
of the simplex on $A$, and so $\tH(X_i(R),t) = t^{i}$. Hence
$\beta(R,t)$ is $t^{(n-1)i + n}$. 
The coefficient of this power of $t$ is $\beta_{|R|-(n-1)i - n,R}$. 
Here $|R| = (i+1)n$. So this is homological degree 
$|R|-(n-1)i - n  = i$ and the
$(n-1)i + n = (i+1)n-i$'th linear strand.

Now let $j > in- (i-1)$. We want to show that the coefficient
$\beta_{|R|-j,R}$ 
of $t^j$ in $\beta(R,t)$ is zero when $|R|-j \leq i-1$.
The $Y_{k,max}(R)$ are simplicial complexes on a subset of $R_k$
and similiarly the $Y_{k+1,min}(R)$ are simplicial complexes 
on subsets of $R_{k+1}$.
Therefore $\tH^p(X_k(R), \kr)$ is nonzero only if 
$p \leq \min \{|R_k|,|R_{k+1}|\} - 1$. 
Let $u$ be such that $|R_u|$ has maximal cardinality among the $R_k$.
Then $\prod_{i=1}^{n-1}\tH(X_i(R),t)$ is nonzero only for 
\[ p \leq
\sum_{k \neq u} |R_k| - n + 1 = |R|-|R_u| - n + 1.\]
Hence if the coefficient
of $t^j$ is nonzero in $\beta(R,t)$, then
\[ j \leq |R|- |R_u| + 1 \leq \frac{n-1}{n} |R| + 1. \]
Thus
\[ nj \leq (n-1) |R| + n \leq (n-1)(j+i-1) + n. \]
This gives $j \leq i(n-1) + 1 = in-(i-1)$, contradicting the assumption that
$j > in- (i-1)$. 
\end{proof}

\begin{lemma} \label{lem:MultiRs}
Let $R_1 \leq R_2 \leq \ldots \leq R_n$ be subsets of $P$.
Then 
\[ |R_1| + |R_2| + \cdots + |R_n| = |R_1 \cup \cdots \cup R_n| + 
\sum_{i = 1}^{n-1} |R_i \cap R_{i+1}|. \]
\end{lemma}

\begin{proof} For a pair $R_1 \leq R_2$ this obviously holds.
Now note that $R_1 \leq R_2 \cup \cdots \cup R_n$ and 
$R_1 \cap (R_2 \cup \cdots \cup R_n)$ equals $R_1 \cap R_2$.
Therefore:
\[ |R_1| + |R_2 \cup \cdots \cup R_n| = |R_1 \cup \cdots \cup R_n|
+ |R_1 \cap R_2|. \]
We may now argue by induction on the number of terms in the union.
\end{proof}

We now characterize the Betti numbers of $L(n,P)$ in the highest homological 
degree. Let $p = |P|$ be the cardinality of $P$. Recall that
$L(n,P)$ is a Cohen-Macaulay ideal of codimension $p$ and hence
has projective dimension $p-1$. 

\begin{proposition} \label{pro:MultiHighhom}
The Betti number $\beta_{p-1,R}$ is nonzero iff 
\begin{itemize}
\item $R_1 \leq R_2 \leq \cdots \leq R_n$,
\item $P = R_1 \cup R_2 \cup \cdots \cup R_n$,
\item For each $i = 1, \ldots, n-1$, we have that $\max(R_i) = \min(R_{i+1})$
is a maximal antichain in $P$.
\end{itemize}
In this case the Betti number $\beta_{p-1,R} = 1$. 
\end{proposition}

\begin{proof}
Suppose that $\beta_{p-1,R}$ is nonzero. Then 
by Theorem \ref{The:LPR:BettiLnP} the product
$\prod_{i=1}^{n-1}\tH(X_i(R),t)$ has 
a nonzero term $t^{|R|-p+1-n}$. 
Now each $\beta(Y_{i,max}(R),t)$ lives in degrees
$\geq |\max(R_i) \cap \min(R_{i+1})| - 2$ by Corollary \ref{lem:MultiIntersect}, 
and so 
\[ \beta(X_i(R),t) = t \beta(Y_{i,max}(R),t) \] 
lives in degrees 
$\geq |\max(R_i) \cap \min(R_{i+1})| - 1$.
Thus $\prod_{i=1}^{n-1}\tH(X_i(R),t)$ lives in degrees 
\[ \geq \sum_{i = 1}^{n-1} |\max(R_i) \cap \min(R_{i+1})| -  (n-1), \]
and if it actually has a nonzero term in this degree, then by
Corollary \ref{lem:MultiIntersect}, $\max(R_i) = \min(R_{i+1})$ for every $i$.

By the start of the proof we must then have
\[ |R|-p+1-n \geq \sum_{i = 1}^{n-1} |\max(R_i) \cap \min(R_{i+1})| -  (n-1). \]
This gives 
\[ |R| \geq p + \sum_{i = 1}^{n-1} |\max(R_i) \cap \min(R_{i+1})|. \]
By Lemma \ref{lem:MultiRs} this is only possible if $P$ 
is $R_1 \cup \cdots \cup R_n$ and we have equality above. But then
by the remark a few lines up, also each $ \max(R_i) = \min(R_{i+1})$
for each $i$.
But then each of these
sets must be a maximal antichain (since the union of the $R_i$'s is
$P$). Thus all three conditions in 
the statement must be fulfilled.

Conversely if these statements are fulfilled, we easily see
that each $Y_{i,max}(R)$ is the boundary of the simplex on $\max(R_i)$.
From this we deduce $\beta_{p-1,R} = 1$.
\end{proof}

In \cite[Sec.~7]{Katt} they pose the problem of when 
$L(n,P)$ is a level Cohen-Macaulay ideal. (This means that
the graded Betti numbers in maximal homological degree $|P|-1$ exist
only in one degree.) Here we answer this problem.

\begin{corollary}
The ideal $L(n,P)$ is a level Cohen-Macaulay ideal
iff all maximal antichains in $P$ have the same cardinality. If
this cardinality is $c$ and $p = |P|$, then the nonzero graded Betti number
in maximal homological degree is $\beta_{p-1,p + (n-1)c}$.
\end{corollary}

\begin{proof}
If all maximal antichains have cardinality $c$, then 
by the above proposition any nonzero $\beta_{p-1,R}$ will have
$R$ of cardinality $p + c(n-1)$. On the other hand, if there is an antichain
$D$ of cardinality $d < c$, then let $R_1$ be the poset ideal generated by 
$D$, let $R_2 = R_3 = \cdots = R_{n-1} = D$, and let $R_n$ be the poset filter
generated by $D$. Then $R$ has cardinality $p + d(n-1)$ and $\beta_{p-1,R} = 1$
and so $L(n,P)$ is not level.
\end{proof}

\begin{remark}
In \cite[Cor.~4.7]{Katt} they show that the first linear strand has length $p-1$,
equal to the projective dimension,  if $P$ has a unique maximal or minimal
element. In fact Proposition \ref{pro:MultiHighhom} 
shows that this linear strand has length $p-1$
iff there is some element in $P$ which is comparable to any other element
in $P$. 
\end{remark}

\section{Letterplace resolutions when $P$ has tree structure}
\label{sec:Trees}
In the case the Hasse diagram of $P$ has the form of a rooted
tree, we shall see that there is a rather simple inductive procedure
which enables us to construct the graded Betti table of $L(n,P)$. In the
last example we compute the resolution of the $n$'th letterplace ideal
of the $V$ shaped poset with three nodes. This is the initial ideal
of the Pfaffians of a $(2n+1) \times (2n+1)$ generic skew-symmetric matrix.

In \cite{FlNe} the second and third author study deformations of 
letterplace ideals $L(2,P)$ when the Hasse diagram of $P$ is
a rooted tree. For such posets $P$ and $Q$ 
it is shown in \cite{HeQuAk} 
that $L(P,Q)$ and $L(Q,P)$ are Alexander dual, which is one of the
few cases this holds.

\subsection{The inductive procedure}
Let $a$ be an element of $P$.
The letterplace ideal $L(n,\Pa)$ lives in the polynomial ring
$\kr[x_{[n] \times (\Pa)}]$. If we include this ring in 
$\kr [x_{[n] \times P}]$, this letterplace ideal generates an extended ideal 
$L(n, \Pa)^e$ in $\kr[x_{[n] \times P}]$.
Similarly we have an extended ideal $L([2, \ldots, n],P)^e$ in 
$\kr[x_{[n] \times P}]$.   
The extended ideal $L(n,\Pa)^e$ is naturally included in $L(n,P)$ which is
again naturally included in $L([2,\cdots,n],P)^e$.

When $a$ is a unique minimal  element of $P$, 
there is also a natural multiplication map
\begin{align*} 
& & L([2,\ldots,n],P)^e(-1) & \mto{\cdot x_{1,a}} L(n,P)& 
\end{align*}

For short we write $L(n-1,P)$ for $L([2,\ldots,n],P)$.
Let $\iota$ denote inclusion maps.
\begin{lemma} Suppose the poset $P$ has a unique minimal element $a$.
Then there is an exact sequence
\begin{equation*} 
0 \vpil L(n,P) \xleftarrow{\left[ \begin{matrix} 
\iota, \cdot x_{1,a} \end{matrix} \right ]}
\begin{matrix} L(n, \Pa)^e  \\ \oplus \\ L(n-1,P)^e(-1)
\end{matrix} 
\xleftarrow{\left [ \begin{matrix} \cdot x_{1,a}, \\-\iota 
\end{matrix} \right ]} \\
L(n,\Pa)^e (-1) \vpil 0.
\end{equation*}

Furthermore $\Tor_i^S(-,\kr)$ applied to the right map is zero
for each $i$.
\end{lemma}

\begin{proof}
The above sequence is clearly a complex. Also every generator of $L(n,P)$ 
is in the image of the left map so this map is surjective. The right
map is clearly also injective. So let $u \oplus -v$ be in  the kernel
of the left map. Then $u = x_{1,a} v$. 
We claim that $v$ is in $L(n,\Pa)^e$. Then clearly the image of $v$
by the right map is $u \oplus -v$.

Let $m_{\Gamma \phi}$ divide $u$ where $\phi : [n] \pil \Pa$ is an isotone map. 
Then $m_{\Gamma \phi}$ divides $v$ and so $v$ is in $L(n,\Pa)^e$. 

\medskip
As for the second statement, it is clear that $\Tor_i^S(-,\kr)$ applied
to $L(n,\Pa)^e \vmto{\cdot x_{1,a}} L(n,\Pa)^e(-1)$ is zero.
Furthermore no $R$ with nonzero $\beta_{i,R}$ of the ideal $L([2,\ldots,n],P)^e$ involves
an element $(1,p) \in R$ for some $p \in  P$ since no generator of this
ideal does. 
But every Betti degree of $L(n,\Pa)^e$ must involve such an element, since
every generator does so.
Hence $\Tor_i^S(-,\kr)$ is zero when applied to the right map in
the statement.
\end{proof}

\begin{corollary} \label{ColRecBetti}
The Betti number $\beta_{i,j}(L(n,P))$ equals the sum
\[ \beta_{i,j}(L(n,\Pa)) + \beta_{i,j-1}(L(n-1,P)) + \beta_{i-1,j-1}(L(n,\Pa)).\]
\end{corollary}

\begin{proof}
The Betti number $\beta_{i,j}(L(n,P))$ equals $\Tor_i^S(L(n,P),\kr)_j$.
The above gives an exact sequence
\begin{align*} 
0 &  \pil   \Tor_i^S(L(n-1,P),\kr)_{j-1} \oplus \Tor_i^S(L(n,\Pa)^e, \kr)_j  &
 \pil \Tor_i^S(L(n,P),\kr)_j \\
&  \pil   \Tor_{i-1}^S(L(n,\Pa)^e)_{j-1} \pil 0, & 
\end{align*}
and the statement follows.
\end{proof}

\subsection{Posets with tree structure}
We now assume that the Hasse diagram of $P$ has the structure
of a rooted tree. So there is a unique minimal element $a$ in $P$,
and every interval $[p,q]$ in $P$ is a chain. We can then describe
completely the multigraded Betti numbers of $L(n,P)$.

\begin{proposition} Suppose the Hasse diagram of $P$ is a rooted tree. Let 
$R \sus [n] \times P$ be such that $R_1 \leq R_2 \leq \cdots \leq R_n$
and let $m_i$ be the cardinality of $\max(R_i)$ for $i = 1, \ldots, n-1$. 
Then this multidegree occurs only in the $p = (1 + \sum_{i=1}^{n-1} m_i)$-linear
strand, and the Betti number $\beta_{|R|-p,R}(L(n,P)) = 1$.
\end{proposition}

\begin{proof}
We apply Proposition \ref{The:LPR:BettiLnP}.
We will show that each $Y_{i,max}(R)$ is 
the boundary of an $(m_i-1)$-simplex. Then each $\tH(X_i(R))$ is 
a power $t^{m_i-1}$ and the result follows. 
But from each element of $\min(R_{i+1})$ there is at most one edge
in the bipartite graph whose edge ideal defines $X_i(R)$,
going to an element of $\max(R_i)$, due to the tree structure of $P$.
Since $R_i \leq R_{i+1}$ also every element of 
$\max(R_i)$ is incident to such an edge. Then the complex $Y_{i,\max}(R)$ 
on $\max(R_i)$ is simply the boundary of the simplex on $\max(R_i)$.
\end{proof}
 
\begin{example}
We compute the Betti table of $L(2,P)$ for the poset $P$ below.
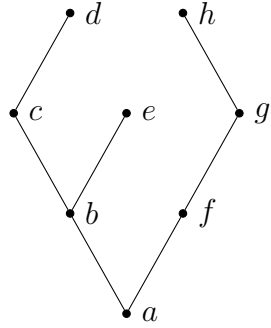
\begin{center}
\begin{tikzpicture}[scale=1, vertices/.style={draw, fill=black, circle, inner sep=1pt}]
              \node [vertices, label=right:{$a$}] (0) at (-0+0,0){};
              \node [vertices, label=right:{$b$}] (5) at (-.75+0,1.33333){};
              \node [vertices, label=right:{$f$}] (1) at (-.75+1.5,1.33333){};
              \node [vertices, label=right:{$c$}] (6) at (-1.5+0,2.66667){};
              \node [vertices, label=right:{$e$}] (4) at (-1.5+1.5,2.66667){};
              \node [vertices, label=right:{$g$}] (2) at (-1.5+3,2.66667){};
              \node [vertices, label=right:{$d$}] (7) at (-.75+0,4){};
              \node [vertices, label=right:{$h$}] (3) at (-.75+1.5,4){};
      \foreach \to/\from in {0/5, 0/1, 5/4, 1/2, 2/3, 5/6, 6/7}
      \draw [-] (\to)--(\from);
\end{tikzpicture}
\end{center}
First we compute the Betti table of the subposet $P_1$,
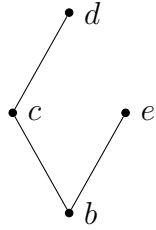
\begin{center}
\begin{tikzpicture}[scale=1, vertices/.style={draw, fill=black, circle, inner sep=1pt}]
              \node [vertices, label=right:{$b$}] (0) at (-0+0,0){};
              \node [vertices, label=right:{$c$}] (1) at (-.75+0,1.33333){};
              \node [vertices, label=right:{$e$}] (3) at (-.75+1.5,1.33333){};
              \node [vertices, label=right:{$d$}] (2) at (-0+0,2.66667){};
      \foreach \to/\from in {0/1, 0/3, 1/2}
      \draw [-] (\to)--(\from);
\end{tikzpicture}
\end{center}
If we remove the point $b$, we end up with subposets $P(c,d)$:
\begin{tikzpicture}[scale=.5, vertices/.style={draw, fill=black, circle, inner sep=1pt}]
              \node [vertices, label=right:{$c$}] (0) at (-0+0,0){};
              \node [vertices, label=right:{$d$}] (1) at (-0+0,1.33333){};
      \foreach \to/\from in {0/1}
      \draw [-] (\to)--(\from);
\end{tikzpicture}
and $P(e)$:
\begin{tikzpicture}[scale=.5, vertices/.style={draw, fill=black, circle, inner sep=1pt}]
              \node [vertices, label=right:{$e$}] (0) at (-0+0,0){};
      \foreach \to/\from in {}
      \draw [-] (\to)--(\from);
\end{tikzpicture}.
The quotient rings by the ideals $L(2,P(c,d))$ and $L(2, P(e))$ have 
resolutions with Betti diagrams
\begin{small}
\begin{align*}
\begin{tabular}{r|ccc}
&0&1&2\\ \hline
\text{0}&1&\text{.}&\text{.}\cr
\text{1}&\text{.}&3&2\cr
\end{tabular}
\qquad
\begin{tabular}{r|cc}
&0&1\\ \hline
\text{0}&1&\text{.}\cr
\text{1}&\text{.}&1 \cr
\end{tabular}
\end{align*}
\end{small}
respectively. These are ideals in rings with distinct variables.
Hence the quotient ring by the ideal generated by the
sum of these ideals has a Betti diagram which is the ``tensor product''
of the above diagrams.

Therefore for the disjoint union of posets $P(c,d)$ and $P(e)$
the Betti diagram of its second letterplace ideal will be:
\begin{small}
\begin{align*}
\begin{tabular}{r|ccc}
&0&1&2\\ \hline
\text{2}&4&2&\text{.}\cr
\text{3}&\text{.}&3&2\cr
\end{tabular}
\end{align*}
\end{small}

Now by Corollary \ref{ColRecBetti}, we have
\begin{small}
\begin{align*}
\operatorname{Betti} (L(2,P_1)) 
& = 
\begin{tabular}{r|ccc}
&0&1&2\\ \hline
\text{2}&4&2&\text{.}\cr
\text{3}&\text{.}&3&2\cr
\end{tabular}
\quad
+
\quad
\begin{tabular}{r|cccc}
&0&1&2&3\\ \hline
\text{2}&\text{.}&4&2&\text{.}\cr
\text{3}&\text{.}&\text{.}&3&2\cr
\end{tabular}
\quad
+
\quad
\begin{tabular}{r|cccc}
&0&1&2&3\\ \hline
\text{2}&4&6&4&1\cr
\end{tabular} \\
& =
\begin{tabular}{r|cccc}
&0&1&2&3\\ \hline
\text{2}&8&12&6&1\cr
\text{3}&\text{.}&3&5&2\cr
\end{tabular}
\end{align*}
\end{small}
The Betti table of the letterplace ideal of the subposet
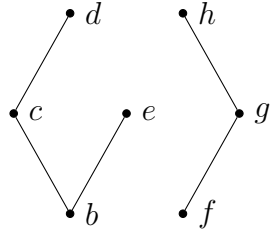
\begin{center}
\begin{tikzpicture}[scale=1, vertices/.style={draw, fill=black, circle, inner sep=1pt}]
              \node [vertices, label=right:{$b$}] (0) at (-.75+0,0){};
              \node [vertices, label=right:{$f$}] (4) at (-.75+1.5,0){};
              \node [vertices, label=right:{$c$}] (1) at (-1.5+0,1.33333){};
              \node [vertices, label=right:{$e$}] (3) at (-1.5+1.5,1.33333){};
              \node [vertices, label=right:{$g$}] (5) at (-1.5+3,1.33333){};
              \node [vertices, label=right:{$d$}] (2) at (-.75+0,2.66667){};
              \node [vertices, label=right:{$h$}] (6) at (-.75+1.5,2.66667){};
      \foreach \to/\from in {0/1, 0/3, 1/2, 4/5, 5/6}
      \draw [-] (\to)--(\from);
\end{tikzpicture}
\end{center}
is again obtained by taking the ``tensor product'' 
of the Betti tables of $L(2,P_1)$ and $L(2,3)$:
\begin{small}
\[
\begin{tabular}{r|ccccccc}
&0&1&2&3&4&5&6\\ \hline
\text{2}&14&20&9&1&\text{.}&\text{.}&\text{.}\cr
\text{3}&\text{.}&51&141&158&90&26&3\cr
\text{4}&\text{.}&\text{.}&18&54&61&31&6\cr
\end{tabular}
\]
\end{small}
Again by using Corollary \ref{ColRecBetti}, $\operatorname{Betti}(L(2,P))$ equals
\begin{small}
\begin{align*}
 & 
\begin{tabular}{r|ccccccc}
&0&1&2&3&4&5&6 \\ \hline
\text{2}&14&20&9&1&\text{.}&\text{.}&\text{.}\cr
\text{3}&\text{.}&51&141&158&90&26&3\cr
\text{4}&\text{.}&\text{.}&18&54&61&31&6\cr
\end{tabular}
\quad
+
\quad
\begin{tabular}{r|cccc cccc}
&0&1&2&3&4&5&6&7\\ \hline
\text{2}&\text{.}&14&20&9&1&\text{.}&\text{.}&\text{.}\cr
\text{3}&\text{.}&\text{.}&51&141&158&90&26&3\cr
\text{4}&\text{.}&\text{.}&\text{.}&18&54&61&31&6\cr
\end{tabular} \\
\\
+\quad &
\begin{tabular}{r|cccc cccc}
&0&1&2&3&4&5&6&7\\ \hline
\text{2}&8&28&56&70&56&28&8&1\cr
\end{tabular}\\
\\
= &
\begin{tabular}{r|cccccccc}
&0&1&2&3&4&5&6&7\\ \hline
\text{2}&22&62&85&80&57&28&8&1\cr
\text{3}&\text{.}&51&192&299&248&116&29&3\cr
\text{4}&\text{.}&\text{.}&18&72&115&92&37&6\cr
\end{tabular}
\end{align*}
\end{small}
\end{example}

\begin{example}
Using this procedure we find that the $n$'th letterplace ideal $L(n,V)$
of the poset $V$: 
\begin{tikzpicture}[scale=.5, vertices/.style={draw, fill=black, circle, inner sep=1pt}]
              \node [vertices, label=right:{$a$}] (0) at (-0+0,-1){};
              \node [vertices, label=right:{$b$}] (1) at (-0.8,0.5){};
              \node [vertices, label=right:{$c$}] (2) at (+0.8,0.5){};
      \foreach \to/\from in {0/1, 0/2}
      \draw [-] (\to)--(\from);
\end{tikzpicture}
has resolution:
\[\begin{tabular}{r|ccc}
&0&1&2\\ \hline
${n}$&$2n+1$&$2n+1$&$1$\cr
${n+1}$&\text{.}&$1$&$1$\cr
{\vdots}&\text{\vdots}&\vdots&\vdots\cr
${2n-1}$&\text{.}&$1$&$1$\cr
\end{tabular}
\]
In fact these letterplace ideals are initial ideals of 
the Pfaffians of a generic $(2n+1) \times (2n+1)$ skew-symmetric matrix,
by \cite[Thm. 5.1]{HeTr}. The variables $X_{i,2n+2-i}$ in loc.cit. correspond
to our variables $x_{a,i}$ for $i = 1, \ldots, n$, 
the variables $X_{i+1,2n+2-i}$ correspond to the $x_{b,i}$ and the 
$X_{i,2n+1-i}$ to the $x_{c,i}$.
\end{example}


\bibliographystyle{amsplain}
\bibliography{Bibliography}

\end{document}